\documentclass[12pt,reqno]{amsart}
\usepackage{amsmath}
\usepackage{amssymb}
\usepackage{amsfonts,color}
\numberwithin{equation}{section}

\newcommand{\DD}{\mathbb{D}}
\newcommand{\NN}{\mathbb{N}}
\newcommand{\RR}{\mathbb{R}}
\newcommand{\TT}{\mathbb{T}}
\newcommand{\ZZ}{\mathbb{Z}}

\newcommand{\cC}{{\mathcal{C}}}
\newcommand{\cD}{{\mathcal{D}}}
\newcommand{\cZ}{{\mathcal{Z}}}
\newcommand{\cM}{{\mathcal{M}}}
\newcommand{\cP}{{\mathcal{P}}}
\newcommand{\cX}{{\mathcal{X}}}

\renewcommand{\hat}{\widehat}

\DeclareMathOperator{\Span}{Span}

\DeclareMathSymbol{\subsetneqq}{\mathbin}{AMSb}{36}

\DeclareMathOperator{\Lip}{Lip}

\newtheorem*{thm}{Theorem}
\newtheorem{theorem}{Theorem}
\newtheorem{lem}[theorem]{{\bf Lemma}}
\newtheorem*{coro}{{\bf Corollary}}

\title[Cyclicity in the harmonic Dirichlet space]
{Cyclicity in the harmonic Dirichlet space}

\thanks{Research of OEF partially supported by Hassan II Academy of Science and Technology.
Research of TR supported by NSERC and the Canada Research Chairs program}

\author[Abakumov]{E. Abakumov}
\address{LAMA, UMR CNRS 8050. Universit\'e Paris--Est\\ 5 boulevard Descartes\\ Champs--sur--Marne\\77454 Marne--la--Vall\'ee cedex 2\\France}
\email{evgueni.abakoumov@u-pem.fr}

\author[El-Fallah]{O. El-Fallah}
\address{Laboratoire Analyse et Applications URAC/03,  Mohammed
V University in Rabat, B.P. 1014 Rabat, Morocco}
\email{elfallah@fsr.ac.ma}

\author[Kellay]{K. Kellay}
\address{IMB\\Universit\'e de  Bordeaux \\
351 cours de la Lib\'eration\\33405 Talence \\France}
\email{kkellay@math.u-bordeaux1.fr}

\author[Ransford]{T. Ransford}
\address{D\'epartement de math\'ematiques et de statistique, Universit\'e Laval\\  
1045 avenue de la M\'edecine, Qu\'ebec (QC), Canada G1V 0A6}
\email{ransford@mat.ulaval.ca}

\begin{document}

\begin{abstract} 
The harmonic Dirichlet space is the Hilbert space of functions $f\in L^2(\TT)$ such that
$$
\|f\|_{\cD(\TT)}^2:=\sum_{n\in\ZZ}(1+|n|)|\hat{f}(n)|^2<\infty.
$$
We give sufficient conditions for $f$  to be cyclic in $\cD (\TT)$, 
in other words, for $\{\zeta ^nf(\zeta):\ n\geq 0\}$ 
to span a dense subspace of $\cD(\TT)$.
\end{abstract}

\maketitle


\section{Introduction}

Let $\cX$ be a topological linear space of complex functions on the unit circle $\TT$ 
such that the shift operator  $S$, given by 
$$
S(f)(\zeta):=\zeta f(\zeta),\qquad f\in \cX,
$$ 
is an isomorphism of $\cX$ onto itself.
 
A closed subspace $\cM$ of $\cX$ is called {\it invariant} if  $S(\cM )\subset \cM$. 
It is said to be {\it $1$-invariant} (or {\it simply invariant}) for $S$ if $S(\cM)\subsetneq \cM$,  
and it is called {\it $2$-invariant} (or {\it doubly invariant}) if $S(\cM )= \cM$. 
The latter condition  is equivalent to the invariance of $\cM$ 
under multiplication by both $\zeta$ and $\overline\zeta$.
  
Let $\ZZ$ denote the integers and let $\NN:=\{n\in\ZZ:n\ge0\}$. 
Given $f\in \cX$,  we write 
\begin{align*}
[f]_{\NN}&:=\overline{\Span}^{\cX}\{z^nf : n\in\NN\},\\
[f]_{\ZZ}&:=\displaystyle\overline{\Span}^{\cX}\{z^nf:n\in\ZZ\}.
\end{align*}
A function $f$ is said to be  {\it $1$-invariant} if the space $[f]_{\NN}$ is $1$-invariant for $S$.  
We say that a function $f\in \cX$ is {\it cyclic } (resp.\  {\it bicyclic}) for $\cX$ 
if $[f]_\NN=\cX$ (resp.\ $[f]_\ZZ=\cX$). 
 
Let us begin with the classical case, namely $\cX=L^2(\TT)$.  
By a well-known theorem of Wiener,
the $2$-invariant subspaces have the form 
$$
\cM= \{f\in L^2(\TT) : 
f=0 \textrm{~a.e.\ on~}  \TT\setminus \sigma\},
$$ 
where $\sigma $ is a Borel subset of $\TT$ (see e.g.\ \cite[p.8, Theorem~1.2.1]{Ni2}). 
It follows from Szeg\H o's infimum theorem
that a function $f$ is $1$-invariant in $L^2(\TT)$ if and only if 
$\log|f|\in L^1(\TT)$ (see e.g.\ \cite[p.12, Corollary~4]{Ni1}). 

For $\cX=\cC^\infty(\TT)$, 
Makarov \cite {M1} gave a complete description of the invariant subspaces of $S$. 
He also obtained the following characterization of $1$-invariant functions of $\cC^\infty(\TT)$.

\begin{thm}[Makarov \protect{\cite[p.3]{M1}}]  
A function $f$ is  $1$-invariant  in $\cC^\infty(\TT)$ if and only if $\log|f|\in L^1(\TT)$.
\end{thm} 

The case $\cX=\cC^n(\TT)$ ($n\geq 1$) is more complicated, 
and no characterization of  $1$-invariant  functions is known 
(\cite[p.3]{M1}, see also \cite[Theorem~1.3]{M2} or \cite[Theorem~5]{M3}).
 
We shall  focus our attention on the {\it harmonic Dirichlet space} $\cD(\TT)$. 
This is the set of functions $f\in L^2(\TT)$  whose Fourier coefficients satisfy
$$ 
\cD(f):=\sum_{n\in\ZZ}|\hat{f}(n)|^2|n|<\infty.
$$
It becomes a Hilbert space if endowed with the norm $\|\cdot\|_{\cD(\TT)}$, given by
$$
\|f\|_{\cD(\TT)}^2:=\|f\|_{L^2(\TT)}^2+\cD(f)=
\sum_{n\in\ZZ}|\hat{f}(n)|^2(1+|n|).
$$

According to Douglas' formula \cite{D}, we have
$$ 
\cD(f)=\frac{1}{4\pi^2}\iint_{\TT^2}
\frac{|f(\zeta)-f(\zeta')|^2}{|\zeta-\zeta'|^2}\,|d\zeta'|\,|d\zeta|.
$$
This can also be written
$$
\cD(f)=\frac{1}{2\pi}\int_{\TT} \cD_\zeta(f)\,|d\zeta|,
$$  
where $\cD_\zeta(f)$ is the so-called local Dirichlet integral  of $f$ at $\zeta$, given by 
$$
\cD_\zeta(f):=
\frac{1}{2\pi}\int_{\TT}\frac{|f(\zeta)-f(\zeta')|^2}{|\zeta-\zeta'|^{2}}{|d\zeta'|},\qquad \zeta\in\TT.
$$

As $\cD(\TT)$ is dense in $L^2(\TT)$, 
if a function $f$ is cyclic for $\cD(\TT)$, 
then  it is also  cyclic for $L^2(\TT)$. 
So by Szeg\H o's theorem  we have
$$
f \text{ is cyclic for }\cD(\TT) \Rightarrow \int_\TT\log |f(\zeta)|\,|d\zeta| =-\infty.
$$  

In \cite{RRS}, Ross, Richter and Sundberg gave a complete characterization 
of the $2$-invariant subspaces $\cM$ of $\cD(\TT)$ in terms of their zero sets.  
In order to state their result, 
we need to introduce the notion of  logarithmic capacity (see  for instance \cite[\S2.4]{EKMR}). 

The {\it energy} of  a Borel probability measure $\mu$ on $\TT$ is defined by 
$$
I(\mu):=
\iint_{\TT^2} \log \frac{1}{|\zeta-\zeta'|}d\mu(\zeta)d\mu(\zeta')=\sum_{n=1}^{\infty}\frac{|\widehat{\mu}(n)|^2}{n}.
$$ 
Then we define the {\it logarithmic capacity} of a Borel subset $E$ of $\TT$ by 
$$c(E):=1/\inf\{ I(\mu): \mu\in \cP(E)\},$$
where $\cP(E)$ denotes the set of  probability measures supported on a compact subset of $E$. 
We say that a property holds {\it quasi--everywhere (q.e.)} 
if it holds everywhere outside a set of logarithmic capacity zero. 

It is known that, if $f\in\cD(\TT)$, 
then  the radial limit of  the Poisson integral of $f$ exists q.e.\ 
and is equal to $f$ a.e.\ (for more details we refer to  \cite{RRS} and the references therein). 
In the sequel, $f$ will denote this limit, and will therefore be defined q.e.\ on $\TT$.
We shall write $\cZ(f)$ for the zero set of $f$, namely
$$
\cZ(f):=\{\zeta\in \TT : f(\zeta)=0\}.
$$
Note that this set is defined up to sets of logarithimic capacity zero.
 
\begin{thm}[Richter--Ross--Sundberg \cite{RRS}]\label {RRS}   
$\cM$ is  a $2$-invariant subspace of $\cD(\TT)$ if and only if 
there exists a measurable set $E\subset \TT$ such that
$$
\cM=\cD_E:=\{f\in \cD(\TT): f|E=0 \text{~q.e.}\}.
$$
\end{thm}

Note that the problem of characterization of $1$-invariant subspaces of $\cD(\TT)$ remains open.  
It was proved in \cite{ARR,R} that, for each $n\in \NN\cup\{\infty\}$, 
there exists an invariant subspace $\cM$ of $\cD(\TT)$ such that $\dim(\cM/S(\cM))=n$.
This suggests that the lattice of $1$-invariant subspaces has a very complicated structure.

As a direct consequence of the Richter--Ross--Sundberg theorem, 
we obtain the following necessary conditions for cyclicity in $\cD(\TT)$.

\begin{theorem} \label{th1}
If $f$ is  cyclic for $\cD(\TT)$, then  
$$ 
\int_\TT\log |f(\zeta)|\,|d\zeta| =-\infty 
\quad \text{ and } \quad 
c(\cZ(f))=0.
$$
\end{theorem}

Our goal in this paper  is to give sufficient conditions for a function $f\in\cD(\TT)$ to be cyclic. 

For $\beta\in(0,1]$,
we shall denote by $\Lip_\beta(\TT)$  the set of functions $f$ continuous on $\TT$ such that
$$
\|f\|_{\Lip_\beta(\TT)}:=
\|f\|_{\cC(\TT)}
+ \sup_{\zeta,\zeta'\in\TT}\frac{|f(\zeta)-f(\zeta')|}{|\zeta-\zeta'|^{\beta}}<\infty.
$$
For $\alpha\in(0,1)$,
we set 
$$
\cC^{1+\alpha}(\TT)
:=\{f\in C^1(\TT): f'\in\Lip_\alpha(\TT)\}.
$$
Of course, if $f$ belongs to $\Lip_\beta(\TT)$ or $C^{1+\alpha}(\TT)$, 
then $\cZ(f)$ is closed in~$\TT$.

We shall establish the following result.

\begin{theorem}\label{th2}
Let $f\in \cD(\TT)$ such that  
$|f| \in\cC^{1+\alpha}(\TT)$, where $\alpha\in(0,1)$. 
Suppose further that $\log|f|\notin L^1(\TT)$.
Then 
$[f^2]_{\NN}= \cD_{\cZ(f)}$.
\end{theorem}

Combining Theorems~\ref{th1} and \ref{th2}, 
we deduce 

\begin{coro}\label{co1}
Let $f\in \cD(\TT)$ such that  
$|f| \in\cC^{1+\alpha}(\TT)$, where $\alpha\in(0,1)$. 
Then the  following assertions are equivalent:
\begin{enumerate}
\item $f^2$ is cyclic for $\cD(\TT)$;
\item $ \log |f|\notin L^1(\TT)$ and $c(\cZ(f))=0$.
\end{enumerate}
\end{coro}

A closed set $E\subset \TT$ is said to be a {\it Carleson set} 
(and we write $E \in (C)$) if 
$$
\int _{\TT}\log \frac{1}{d(\zeta,E)}\,|d\zeta| <\infty.
$$
For background information on Carleson sets, see e.g.\ \cite[\S4.4]{EKMR}.
Note that, if $f\in\Lip_\beta (\TT)$ and $\cZ(f)\notin(C)$,  then $\log|f|\notin L^1(\TT)$. 

It is known that $\Lip_\beta(\TT)\subset\cD(\TT)$ 
if and only if $\beta >1/2$. 
The inclusion $\Lip_\beta(\TT)\subset \cD(\TT)$ 
for  $\beta>1/2$
can easily be obtained from Douglas' formula.

We shall establish the following theorem.

\begin{theorem}\label{th3}
Let $f\in Lip_{\beta}(\TT)$, where $\beta\in(\frac{1}{2},1]$. 
If $\cZ(f) \notin (C)$, then
$[f]_{\NN}=\cD _{\cZ(f)}$.
If furthermore $c(\cZ(f))=0$, 
then $f$ is cyclic for  $\cD(\TT)$. 
\end{theorem}


\section{Proof of Theorem \ref {th2}}
 
For the proof of Theorem~\ref{th2},
we shall need the following standard result.

\begin{lem}\label{lemmul} 
Let $f\in\cD(\TT)$. 
The  following assertions are equivalent:
\begin{enumerate}
\item $[f]_\NN=[f]_{\ZZ}$;
\item $f \in [Sf]_\NN$;
\item $\inf\{\|pf\|_{\cD(\TT)}: p\in H^\infty,~  pf\in\cD(\TT)  \text{~and~}   p(0)=1\}=0.$
\end{enumerate}
\end{lem}
\begin{proof}  Since $S$ is invertible, (1) and (2) are equivalent. 

If $f\in[Sf]_\NN$,
then there is a sequence $(p_n)$ of polynomials  
such that $p_n(0)=1$ and  $\|(1-p_n) f\|_{\cD(\TT)}\to 0$. 
This proves that (2) implies (3).

Finally, suppose that (3) holds. 
Let  $(p_n) \subset H^\infty$ be a sequence such that 
$p_n(0) =1$, $p_nf \in \cD (\TT)$ 
and $\|p_nf\|_{\cD(\TT)}\to 0$. 
Writing $p_n = 1-zq_n$,
by \cite[Proposition 3.4]{RRS} 
we have $zq_n f \in [Sf]_{\NN}$. 
Since $zq_nf$ converges to $f$,
it follows that $f\in[Sf]_\NN$,
so that (2) holds.
\end{proof}

We shall also need the following result, 
which is a special case of a theorem due to  
Carleson--Jacobs--Havin--Shamoyan \cite[Theorem~6.1]{B}.

\begin{lem}\label{tam}  
Let $F$ be an outer function on $\DD$ that is continuous on $\overline{\DD}$. 
If $|F|\in\cC^{1+\alpha}(\TT)$, where $\alpha\in(0,1)$,
 then  $F\in \Lip_{({1+\alpha})/{2}}(\overline{\DD})$. 
 Furthermore, the Lipschitz constant associated to $F$ on $\overline{\DD}$  
 depends only on the Lipschitz constants and bounds for the derivatives of $|F|$ on $\TT$.
\end{lem}

\begin{proof}[Proof of Theorem~\ref{th2}]
Let  $p_\epsilon$ be the outer function such that
$$
|p_\epsilon(\zeta)|
=\frac{e^{-M_\epsilon}}{|f(\zeta)|+\epsilon}
\quad \text{a.e. on~}\TT,
$$
where the constant $M_\epsilon$ is chosen so that 
$p_\epsilon(0)=1$. Thus
$$
M_\epsilon 
=\int_{\TT}\log\Bigl(\frac{1}{|f(\zeta)|+\epsilon}\Bigr)\,\frac{|d\zeta|}{2\pi},
$$
and since $\log|f|\notin L^1(\TT)$,
it follows that $M_\epsilon\to\infty$ as $\epsilon\to0^+$.
We are going to prove that 
$$
\lim_{\epsilon\to 0^+}\|p_\epsilon f^2\|_{\cD(\TT)}=0.
$$
If this holds, then by Lemma~\ref{lemmul} we have 
$[f^2]_\NN=[f^2]_{\ZZ}$,
and since clearly $\cZ(f^2)=\cZ(f)$, 
we can apply the  Richter--Ross--Sundberg theorem to obtain the desired result. 
 
We have
$$ 
\|p_\epsilon f^2\|_{\cD(\TT)}^{2}
=\|p_\epsilon f^2\|_{L^2(\TT)}^2+\cD(p_\epsilon f^2).
$$
For the first term, we have
$$
\|p_\epsilon f^2\|_{L^2(\TT)}^2
= \int_\TT\frac{e^{-2M_\epsilon}|f|^{4}}
{(|f|+\epsilon)^2}\frac{|d\zeta|}{2\pi}
\leq e^{-2M_\epsilon}\|f^{2}\|_{L^2(\TT)}^{2}\to0
\quad\text{as~}\epsilon\to0^+.
$$
The second term we estimate using Douglas' formula,
namely
$$
\cD(p_\epsilon f^2)= 
\frac{1}{4\pi^2}\iint_{\TT^2}\frac{|(p_\epsilon f^2)(\zeta)-(p_\epsilon f^2)(\zeta')|^{2}}{|\zeta-\zeta'|^2}\,|d\zeta'|\,|d\zeta|.
$$
Let 
$$
\Gamma:=
\{(\zeta,\zeta')\in\TT^2: |f(\zeta')|\leq |f(\zeta)|\}.
$$ 
Then, by symmetry,
$$
\cD(p_\epsilon f)
= 2\frac{1}{4\pi^2}\iint_{\Gamma}\frac{|(p_\epsilon f)(\zeta)-(p_\epsilon f)(\zeta')|^{2}}{|\zeta-\zeta'|^2}\,|d\zeta'|\,|d\zeta|.
$$
Now, for all $\zeta,\zeta'\in\TT$, we have
\begin{align*}
&|(p_\epsilon f^2)(\zeta)-(p_\epsilon f^2)(\zeta')|^{2}\\
&\quad=|p_\epsilon(\zeta)(f^2(\zeta)-f^2(\zeta'))+f^2(\zeta')(p_\epsilon(\zeta)-p_\epsilon(\zeta'))|^2\\
&\quad\leq 
2|p_\epsilon(\zeta)|^2|f^2(\zeta)-f^2(\zeta')|^2+2|f^2(\zeta')|^2|p_\epsilon(\zeta)-p_\epsilon(\zeta')|^2.
\end{align*}
Hence
$$
\iint_{\Gamma}\frac{|(p_\epsilon f^2)(\zeta)-(p_\epsilon f^2)(\zeta')|^{2}}{|\zeta-\zeta'|^2}\,|d\zeta'|\,|d\zeta|\leq 2A_\epsilon+2B_\epsilon,$$
where 
$$
A_\epsilon:=
\iint_{\Gamma}|p_\epsilon(\zeta)|^2\frac{|f^2(\zeta)-f^2(\zeta')|^2}{|\zeta-\zeta'|^2}\,|d\zeta'|\,|d\zeta|
$$
and  
$$
B_\epsilon:=
\iint_{\Gamma}|f^2(\zeta')|^2\frac{|p_\epsilon(\zeta)-p_\epsilon(\zeta')|^{2}}{|\zeta-\zeta'|^2}\,|d\zeta'|\,|d\zeta|.
$$

We estimate $A_\epsilon$ directly as follows:
\begin{align*}
A_\epsilon
&=e^{-2M_\epsilon}\iint_{\Gamma}\frac{|f(\zeta)+f(\zeta')|^{2}}{(|f(\zeta)|+\epsilon)^{2}}\frac{|f(\zeta)-f(\zeta')|^{2}}{|\zeta-\zeta'|^2}\,|d\zeta'|\,|d\zeta|\\
&\leq 4e^{-2M_\epsilon}  \iint_{\TT^2}\frac{|f(\zeta)-f(\zeta')|^{2}}{|\zeta-\zeta'|^2}\,|d\zeta'|\,|d\zeta|\\
&\leq 4e^{-2M_\epsilon}4\pi^2\cD(f).
\end{align*}
Hence $A_\epsilon\to0$ as $\epsilon\to0^+$.

To estimate $B_\epsilon$, 
we consider the outer function $F_\epsilon$ such that
$$
|F_\epsilon(\zeta)|=|f(\zeta)|+\epsilon\quad \text{a.e. on~} \TT.
$$
By Lemma~\ref{tam},
since  $|F_\epsilon|\in\cC^{1+\alpha}(\TT)$, 
 we have $F_\epsilon\in\Lip_{(1+\alpha)/2}(\TT)\subset\cD(\TT)$  
and there exists a positive constant $D$, depending only on $|f|$, 
such that $\cD(F_\epsilon)\leq D$ for all $\epsilon\in(0,1)$. We then have 
\begin{align*}
B_\epsilon
&=\iint_{\Gamma}\frac{e^{-2M_\epsilon}|f^2(\zeta')|^2}{(|f(\zeta)|+\epsilon)^2(|f(\zeta')|+\epsilon)^2}\frac{|1/p_\epsilon(\zeta)-1/p_\epsilon(\zeta')|^2}{|\zeta-\zeta'|^2}\,|d\zeta'|\,|d\zeta|\\
&\leq e^{-2M_\epsilon}\iint_{\Gamma}
\frac{|F_\epsilon(\zeta)-F_\epsilon(\zeta')|^2}{|\zeta-\zeta'|^2}\,|d\zeta'|\,|d\zeta|\\
& \leq e^{-2M_\epsilon} 4\pi^2\cD(F_\epsilon).
\end{align*}
Thus $B_\epsilon\to0$ as $\epsilon\to0^+$.
This completes the proof of Theorem~\ref{th2}.
\end{proof}


\section{Proof of Theorem~\ref{th3}}

To prove Theorem~\ref {th3}, we shall need the following additional lemma.

\begin{lem}\label{lip}
Let $f\in \Lip_\beta(\TT)$, where $\beta>1/2$. 
Then, for $\eta \in (0,\frac{2\beta-1}{2\beta})$, we have
$$
\iint_{\TT^2}\frac{|f(\zeta)-f(\zeta')|^{2-2\eta}}{|\zeta-\zeta'|^2}\,|d\zeta'|\,|d\zeta|<+\infty.
$$
\end{lem}

\begin{proof} 
Since  $\beta>1/2(1-\eta)$, we get
$$
\iint_{\TT^2}
\frac{|f(\zeta)-f(\zeta')|^{2-2\eta}}{|\zeta-\zeta'|^2}\,|d\zeta'|\,|d\zeta|
\lesssim 
\iint_{\TT^2}\frac{|d\zeta'|\,|d\zeta|}{|\zeta-\zeta'|^{2(1-(1-\eta)\beta)}}<\infty.\qedhere
$$
\end{proof}

Note that here, and in what follows, we write
$A\lesssim B$  to mean that there is an absolute constant $C$ such that $A \leq CB$.

\begin{proof}[Proof of Theorem~\ref{th3}]
By \cite {RRS}, it suffices to prove that 
$[f]_{\NN}$ is $2$-invariant, 
which is equivalent to proving that 
$f\in [Sf]_{\NN}$. 

Let $\epsilon,\gamma >0$, 
where $\gamma$ will be taken small. 
Let $E$ be a closed subset of $\cZ(f)$ such that $|E|=0$ and $E\notin(C)$. 
Let $p_\epsilon$ be the outer function satisfying  
$$
|p_\epsilon(\zeta)|
=\frac{e^{-M_\epsilon}}{(d(\zeta,E)^\gamma+\epsilon)^{1/2}}
\quad  \text{a.e. on~}\TT,
$$
where the constant $M_\epsilon$ is chosen so that 
$p_\epsilon(0)=1$. Thus
$$
M_\epsilon:=
\frac{1}{2}\int_\TT \log \Bigl(\frac{1}{d(\zeta,E)^\gamma+\epsilon }\Bigr)\,\frac{|d\zeta|}{2\pi},
$$
and since $E\notin(C)$, it follows that $M_\epsilon\to\infty$ as $\epsilon\to0^+$.
By Lemma~\ref{lemmul}, it suffices to prove that 
$$
\lim _{\epsilon \to 0^+}\|p _\epsilon f\|_{\cD(\TT)}=0.
$$

Now
$$ 
\|p_\epsilon f\|_{\cD(\TT)}^{2}
=\|p_\epsilon f\|_{L^2(\TT)}^2+\cD(p_\epsilon f).
$$
 For the first term, we have
$$
\|p_\epsilon f\|_{L^2(\TT)}^{2}\lesssim e^{-2M_\epsilon} 
\int_\TT\frac{d(\zeta,E)^{2\beta}}{d(\zeta,E)^\gamma}\,|d\zeta|.
$$
Thus $\|p_\epsilon f\|_{L^2(\TT)}\to0$ as $\epsilon \to 0^+$, 
provided that $\gamma <2\beta$.
 
For the second term we again use Douglas' formula,
namely
$$
\cD(p_\epsilon f)
=\frac{1}{4\pi^2}\iint_{\TT^2}\frac{|(p_\epsilon f)(\zeta)-(p_\epsilon f)(\zeta')|^{2}}{|\zeta-\zeta'|^2}\,|d\zeta'|\,|d\zeta|.
$$
Let 
$$
\Gamma:=\{(\zeta,\zeta ')\in\TT^2:  
d(\zeta ',E)\leq  d(\zeta ,E) \}.
$$
Arguing as in the proof of Theorem~\ref{th2}, we have
\begin{align*}
\cD (p_\epsilon f)
& \lesssim \iint _\Gamma |p_\epsilon(\zeta) |^2\frac{|f(\zeta)-f(\zeta ')|^2}{|\zeta -\zeta '|^2}|d\zeta||d\zeta '|\\
& \quad+\displaystyle  \iint _\Gamma |f (\zeta ') |^2\frac{|p(\zeta)-p(\zeta ')|^2}{|\zeta -\zeta '|^2}|d\zeta||d\zeta '|\\
&= A_\epsilon +B_\epsilon,~\text{say}.
\end{align*}

To estimate $A_\epsilon$,
let  $\eta \in (0,(2\beta-1)/(2\beta))$ and let  
$\gamma<2\beta\eta$. Then 
\begin{align*}
A_\epsilon 
&=  e^{-2M_\epsilon} \iint _\Gamma \Bigl(\frac{|f(\zeta)-f(\zeta ')|^{2\eta}}{d(\zeta,E)^\gamma+\epsilon}\Bigr)\Bigl(\frac{|f(\zeta)-f(\zeta ')|^{2-2\eta}}{|\zeta -\zeta '|^2}\Bigr)\,|d\zeta|\,|d\zeta '|\\
&\leq  2Ce^{-2M_\epsilon}\iint _{\TT^2} 
\frac{d(\zeta, E)^{2\beta \eta}}{d(\zeta,E)^\gamma}\Bigl(\frac{|f(\zeta)-f(\zeta ')|^{2-2\eta}}{|\zeta -\zeta '|^2}\Bigr)\,|d\zeta|\,|d\zeta'|\\
&\leq 2C^{-2M_\epsilon}\iint_{\TT^2}\frac{|f(\zeta)-f(\zeta ')|^{2-2\eta}}{|\zeta -\zeta'|^2}\,|d\zeta|\,|d\zeta '|, 
\end{align*}
where $C$ depends only on $f$, and the double integral is finite, thanks to Lemma~\ref{lip}.
Hence $A_\epsilon\to0$ as $\epsilon \to 0^+$.  

To estimate $B_\epsilon$, we introduce the outer function $F_\epsilon$ satisfying 
$$
|F_\epsilon(\zeta)|= d(\zeta,E)^\gamma+\epsilon 
\quad \text{a.e. on~}\TT.
$$
By the Carleson--Richter--Sundberg formula 
\cite[Theorem 7.4.2]{EKMR}, we have
$$
\cD_\zeta(F_\epsilon)=
\int_\TT \frac{|F_\epsilon(\zeta)|^2-|F_\epsilon(\zeta')|^2-2|F_\epsilon(\zeta')|\log |F_\epsilon(\zeta)/F_\epsilon(\zeta')|}{|\zeta-\zeta'|^2}\,\frac{|d\zeta'|}{2\pi}.
$$
Therefore
\begin{align*}
B_\epsilon 
&\lesssim e^{-2M_\epsilon} \iint_\Gamma \frac{d(\zeta',E) ^{2\beta}}{(d(\zeta,E)^\gamma+\epsilon)(d(\zeta ',E)^\gamma
+ \epsilon)}\frac{|F_\epsilon(\zeta)-F_\epsilon(\zeta')|^2}{|\zeta-\zeta'|^2}\,|d\zeta|\,|d\zeta '|\\
& \lesssim  e^{-2M_\epsilon}\iint_{\TT^2} \frac{|F_\epsilon(\zeta)-F_\epsilon(\zeta')|^2}{|\zeta -\zeta'|^2}d(\zeta',E)^{2(\beta-\gamma)}\,|d\zeta|\,|d\zeta'|\\
& \lesssim e^{-2M_\epsilon}\int_\TT \cD_{\zeta}(F_\epsilon) d(\zeta,E)^{2(\beta-\gamma)}\,|d\zeta |\\
&\lesssim  e^{-2M_\epsilon} \iint_{\TT^2}\Bigl(\frac{|F_\epsilon(\zeta)|^2-|F_\epsilon(\zeta')|^2-2|F_\epsilon(\zeta')|\log |F_\epsilon(\zeta)/F_\epsilon(\zeta')|}{|\zeta-\zeta'|^2}\Bigr) \\
& \qquad\qquad\qquad\times 
\Bigl(d(\zeta,E)^{2(\beta-\gamma)} +d(\zeta',E)^{2(\beta-\gamma)}\Bigr)\,|d\zeta|\,|d\zeta'|.
\end{align*}
Exchanging the roles of $\zeta$ and $\zeta'$, and taking the average, we obtain 
\begin{align*}
 B_\epsilon 
 &\lesssim  e^{-2M_\epsilon} \iint_{\TT^2}\Bigl(\frac{(|F_\epsilon(\zeta)|^2-|F_\epsilon(\zeta')|^2)\log |F_\epsilon(\zeta)/F_\epsilon(\zeta')|}{|\zeta-\zeta'|^2}\Bigr) \\
& \qquad\qquad\qquad\times 
\Bigl(d(\zeta,E)^{2(\beta-\gamma)} +d(\zeta',E)^{2(\beta-\gamma)}\Bigr)\,|d\zeta|\,|d\zeta'|.
\end{align*}
Thus
\begin{equation}\label{eqB1}
B_\epsilon\lesssim 
e^{-2M_\epsilon}\iint_{\TT^2} 
\frac{\delta^\gamma-\delta'^\gamma}{|\zeta-\zeta'|^2}\log \Bigl(\frac{\delta^\gamma+\epsilon}{\delta'^\gamma+\epsilon}\Bigr)(\delta^{2(\beta-\gamma)}+\delta'^{2(\beta-\gamma)})\,|d\zeta|\,|d\zeta'|,
\end{equation}
where $\delta:=d(\zeta,E)$ and $\delta':=d(\zeta',E)$.

Let $(I_j)$ be the connected components of $\TT\setminus E$, and set 
$$
N_E(t):=2\sum_{j}1_{\{|I_j|>2t\}},
\qquad 0<t<1.
$$
Then, for every measurable function 
$\Omega:[0,\pi]\to\RR^+$, we have 
$$
\int_\TT\Omega(d(\zeta,E))\,|d\zeta|
=\int_0^\pi \Omega(t)N_E(t)\,dt.
$$
Using  similar ideas to those in \cite{EKR0,EKR}, 
we obtain
\begin{align*}
J 
&:= \iint_{\TT^2}\frac{\delta^\gamma-\delta'^\gamma}{|\zeta-\zeta'|^2}\log\Bigl(\frac{ \delta^\gamma+\epsilon}{\delta'^\gamma+\epsilon}\Bigr)(\delta^{2(\beta-\gamma)}+\delta'^{2(\beta-\gamma)})\,|d\zeta|\,|d\zeta'|\\
&\lesssim \int_0^\pi\int_0^\pi \frac{((s+t)^\gamma-t^\gamma)}{s^2}\log \Bigl(\frac{(s+t)^\gamma+\epsilon }{t^\gamma+\epsilon}\Bigr) (t+s)^{2(\beta-\gamma)}N_E(t)\,ds\,dt\\
& \lesssim \int_0^\pi\int_0^t \frac{((s+t)^\gamma-t^\gamma)\log[(s+t)^\gamma /t^\gamma]}{s^2} (t+s)^{2(\beta-\gamma)}N_E(t)\,ds\,dt\\
& \quad + \int_0^\pi \int_{t}^{\pi} \frac{((s+t)^\gamma-t^\gamma)}{s^2}\log[1/(t^\gamma+\epsilon)](t+s)^{2(\beta-\gamma)}\,ds \,N_E(t)dt\\
& =  J_1+J_2, \\
\intertext{where}
J_1 &\lesssim  \int_0^\pi t^{2\beta-\gamma-1} \int_{0}^{1}  \frac{((1+x)^\gamma-1)}{x^2} \log (1+x)\,dx \,N_E(t)\,dt\\
& \lesssim \int_0^\pi t^{2\beta-\gamma-1} N_E(t)\,dt= O(1),\\
\intertext{and}
J_2 & \lesssim \int_0^\pi t^{2\beta-\gamma-1}\log[1/(t^\gamma+\epsilon)]\int_{1}^{\pi/t}  \frac{(1+x)^\gamma-1)}{s^2} (1+x)^{2(\beta-\gamma)}\,dx \,N_E(t)dt\\
& \lesssim \int_0^\pi \log[1/(t^\gamma+\epsilon)]N_E(t)\,dt\\
& \lesssim \int_\TT |\log (d(\zeta,E)^\gamma+\epsilon)|\,|d\zeta|=2 M_\epsilon. 
\end{align*}
Thus $J = O(M_\epsilon)$. Combining this with the estimate \eqref{eqB1},  we get 
$$
B_\epsilon\lesssim M_\epsilon e^{-2M_\epsilon}.
$$
Hence $B_\epsilon\to0$ as $\epsilon\to0^+$.
This completes the proof of Theorem~\ref{th3}.
\end{proof}


\section{Concluding remarks} 

{\bf1.} In order to produce cyclic functions for $\cD(\TT)$ using Theorem~\ref{th3}, 
we need to construct closed subsets 
$E\subset\TT$ such that $E\notin(C)$ and $c(E)=0$. 
An easy example can be given by countable sets. Indeed, taking
$E_\beta:=\{e^{i/(\log n)^\beta}: n\geq 2\}$ 
with $\beta\leq 1$ provides such an example. 
Using Cantor-type sets, 
it is also possible to construct  perfect sets $E$ such that $E\notin(C)$ and $c(E)=0$.

{\bf2.} One can consider weighted harmonic Dirichlet spaces 
instead of the classical harmonic Dirichlet space. 
More precisely, given $\alpha\in[0,1)$,  
the {\it weighted harmonic Dirichlet space} 
$\cD_\alpha(\TT)$ is the space  of functions 
$f\in L^2(\TT)$  such that
$$ 
\|f\|_{\cD_\alpha(\TT)}^{2}
:=\sum_{n\in\ZZ}|\widehat{f}(n)|^2(1+|n|)^{1-\alpha}<\infty.
$$
We define the {\it $\alpha$-capacity} of a Borel subset $E\subset\TT$ by 
$$
c_\alpha(E)=1/\inf\{I_\alpha(\mu): \mu\in \cP(E)\},
$$
where $\cP(E)$ is the set of all probability measures 
supported on a compact subset of $E$ and 
$I_\alpha(\mu):=\sum_{n\geq1}{|\hat{\mu}(n)|^2}/n^{1-\alpha}$
is the {\it $\alpha$-energy} of $\mu$. 
We say that a property holds {\it $c_\alpha$-quasi-everywhere}  
if it holds everywhere outside a set of 
$c_\alpha$-capacity zero. 

It is well known that 
$\Lip_{\beta}(\TT)\subset\cD _\alpha (\TT)$ 
if and only $\beta >(1-\alpha)/2$. 
Theorem~\ref{th3} may be extended to show that, 
if $f\in\Lip_{\beta}(\TT)$, 
where $\beta\in((1-\alpha)/2,1]$, and if $\cZ(f)\notin(C)$,  
then 
$$
[f]_{\NN} 
=\{g\in \cD(\TT): g|_{\cZ(f)}=0~ \text{$c_\alpha$-quasi-everywhere}\}.
$$ 

 {\bf 3.} One can equally well consider the holomorphic Dirichlet space,
 namely $\cD:=\{f\in\cD(\TT):\hat{f}(n)=0~(n<0)\}$.
Here too the problem of characterizing the cyclic functions is still open. 
For more on this topic,
see e.g.\ \cite[Chapter~9]{EKMR}.

\section*{Acknowledgement}
Part of the research for this article was carried out while the authors were visiting the University of Bordeaux to participate at the Conference on Harmonic Analysis, Function Theory, Operator Theory and Applications in honor of Jean Esterle. The authors thank the Institut de Mathématiques de Bordeaux for its support and hospitality.

\end{document}